\documentclass[11pt,reqno,tbtags,draft]{amsart}
\usepackage{amssymb, color}
\usepackage{url}
\usepackage[square,numbers]{natbib}
\bibpunct[, ]{[}{]}{;}{n}{,}{,}

\title
{Moment convergence of balanced P{\'o}lya processes}

\date{21 June, 2016}

\author{Svante Janson}
\thanks{SJ partly supported by the Knut and Alice Wallenberg Foundation}
\address{Department of Mathematics, Uppsala University, PO Box 480,
SE-751~06 Uppsala, Sweden}
\email{svante.janson@math.uu.se}
\newcommand\urladdrx[1]{{\urladdr{\def~{{\tiny$\sim$}}#1}}}
\urladdrx{http://www2.math.uu.se/~svante/}

\author{Nicolas Pouyanne}
\address{Laboratoire de Math\'ematiques de Versailles, UVSQ, CNRS, Universit\'e Paris-Saclay, 78035 Versailles, France.}
\email{nicolas.pouyanne@uvsq.fr}
\urladdrx{http://pouyanne.perso.math.cnrs.fr}

\subjclass[2010]{60C05}

\numberwithin{equation}{section}

\renewcommand\le{\leqslant}
\renewcommand\ge{\geqslant}

\allowdisplaybreaks

\theoremstyle{plain}
\newtheorem{theorem}{Theorem}[section]
\newtheorem{lemma}[theorem]{Lemma}
\newtheorem{proposition}[theorem]{Proposition}
\newtheorem{corollary}[theorem]{Corollary}

\theoremstyle{definition}

\newtheorem{problem}[theorem]{Problem}
\newtheorem{remark}[theorem]{Remark}

\theoremstyle{remark}

\newenvironment{romenumerate}[1][-10pt]{
\addtolength{\leftmargini}{#1}\begin{enumerate}
 }{\end{enumerate}}

\newcounter{oldenumi}
{\setcounter{oldenumi}{\value{enumi}}
\begin{romenumerate} \setcounter{enumi}{\value{oldenumi}}}
{\end{romenumerate}}

\newcounter{thmenumerate}

\newcounter{xenumerate}   

\newcommand\xfootnote[1]{\unskip\footnote{#1}$ $} 

\newcommand\pfitemx[1]{\par#1:}
\newcommand\pfitemref[1]{\pfitemx{\ref{#1}}}

\newcommand{\refT}[1]{Theorem~\ref{#1}}
\newcommand{\refC}[1]{Corollary~\ref{#1}}
\newcommand{\refL}[1]{Lemma~\ref{#1}}
\newcommand{\refR}[1]{Remark~\ref{#1}}
\newcommand{\refS}[1]{Section~\ref{#1}}

\begingroup
  \divide\count255 by 60
  \multiply\count255 by -60
  \advance\count255 by \time
  \ifnum \count255 < 10 \xdef\klockan{\the\count1.0\the\count255}
  \else\xdef\klockan{\the\count1.\the\count255}\fi
\endgroup

\newcommand{\sumin}{\sum_{i=1}^n}
\newcommand{\sumis}{\sum_{i=1}^s}

\newcommand{\sumjs}{\sum_{j=1}^s}

\newcommand{\sumks}{\sum_{k=1}^s}

\newcommand\set[1]{\ensuremath{\{#1\}}}
\newcommand\bigset[1]{\ensuremath{\bigl\{#1\bigr\}}}

\newcommand\xpar[1]{(#1)}
\newcommand\bigpar[1]{\bigl(#1\bigr)}

\newcommand\bigabs[1]{\bigl|#1\bigr|}

\newcommand\biggabs[1]{\biggl|#1\biggr|}

\def\rompar(#1){\textup(#1\textup)}    

\newcommand\innprod[1]{\langle#1\rangle}
\newcommand\innprodgl[1]{\innprod{\gl,#1}}

\def\xexp(#1){e^{#1}}

\newcommand\ntoo{\ensuremath{{n\to\infty}}}

\newcommand\punkt{.\spacefactor=1000}    
    
\newcommand\ie{i.e\punkt}
\newcommand\eg{e.g\punkt}

\newcommand\cf{cf\punkt}
\newcommand{\as}{a.s\punkt}

\newcommand\ii{\mathrm{i}}

\newcommand{\tend}{\longrightarrow}
\newcommand\dto{\overset{\mathrm{d}}{\tend}}

\newcommand\bbR{\mathbb R}
\newcommand\bbC{\mathbb C}

\newcommand\bbZ{\mathbb Z}

\newcommand\bbZgeo{\mathbb Z_{\ge0}}

\newcounter{CC}
\newcounter{cc}

\renewcommand\Re{\operatorname{Re}}

\newcommand\E{\operatorname{\mathbb E{}}}
\renewcommand\P{\operatorname{\mathbb P{}}}

\newcommand\Var{\operatorname{Var}}

\newcommand\supp{\operatorname{supp}}

\newcommand\ga{\alpha}
\newcommand\gb{\beta}
\newcommand\gd{\delta}

\newcommand\gam{\gamma}

\newcommand\gl{\lambda}
\newcommand\gL{\Lambda}

\newcommand\gs{\sigma}
\newcommand\gS{\Sigma}

\newcommand\eps{\varepsilon}
\newcommand\ep{\varepsilon}
\renewcommand\phi{\xxx}  

\newcommand\qq{^{1/2}}

\newcommand\lhs{left-hand side}

\newcommand\xij{_{ij}}

\newcommand\spann{\operatorname{span}}
\newcommand\xxn{x_1,\dots,x_n}
\newcommand\tX{\widetilde X}
\newcommand\uu{\mathbf{u}}
\newcommand\bbZss{\bbZgeo^s}
\newcommand\Varx[1]{\Var(#1)}
\newcommand\wprod[1]{\innprod{w,#1}}

\newcommand\id{\mathrm{id}}
\newcommand\xF{S'}
\newcommand\gdIx{\gd_I^*}
\newcommand\gdix{\gd^*_{\set1}}

\newcommand{\Polya}{P\'olya}

\begin{document}

\begin{abstract} 
It is known that in an irreducible 
small P\'olya urn process, the composition of the urn
after
suitable normalization converges in distribution to a normal distribution.
We show that if the urn also is balanced, this normal convergence holds with
convergence of all moments, thus giving asymptotics of (central) moments.
\end{abstract}

\maketitle

\section{Introduction}\label{S:intro}

A \Polya{} urn process is 
defined as follows.
Consider an urn containing balls of different colours, with $s$
possible colours which we label $1\dots,s$. At each time step, we
draw a ball at random from the urn; we then replace it and, if its colour
was $i$, we add $r\xij$ further balls of colour $j$, for each $j=1,\dots,s$.
Here 
\begin{equation}\label{R}
  R:=(r\xij)_{i,j=1}^s
\end{equation}
is a given matrix, called the \emph{replacement matrix}.
The state of the urn at time $n$ is described by a vector
$X_n=(X_{n1},\dots,X_{ns})$, where $X_{nj}$ is the number of balls of colour
$j$. We start with some given (deterministic)
$X_0$, and it is clear that $X_n$ evolves
according to a Markov process.

As usual,
we assume that $r_{ij}\ge0$ when $i\neq j$, but we allow $r_{ii}$
to be negative, meaning removal of balls, provided the urn is
\emph{tenable}, \ie, that it is impossible to get
stuck. 
(See \eqref{ten1}--\eqref{ten2}, and 
see \refR{Rnegative} for an extension that allows some negative $r\xij$.)

Urn processes of this type 
have been studied by many different authors, 
with varying generality, 
going back to \citet{EggPol};
see for example
\citet{SJ154},
\citet{Flajolet-analytic},
\citet{P},
\citet{Mahmoud},
and the further references given there.

In the present paper we study only the \emph{balanced} case, meaning that
the total number of balls added each time is deterministic, \ie, that the
row sums of the matrix \eqref{R} are constant, say $m$; we assume further
that $m>0$. 

We define, for an arbitrary vector $(\xxn)$, $|(\xxn)|:=\sumin|x_i|$.
In particular, the total number of balls in the urn is $|X_n|$.
Note that when the urn is balanced, this number is
deterministic, with $|X_n|=|X_0|+nm$.

In the description above, it is implicit that the numbers $r\xij$ are integers.
However, it has been noted many times that the process is also well-defined
for \emph{real} $r\xij$, 
see \eg{} \cite[Remark 4.2]{SJ154}, \cite{SJ169} and \cite{P} (\cf{} also
\cite{Jirina} for the related case of branching processes);
this can be interpreted as an urn containing a certain amount (mass) of each
colour, rather than discrete balls. We give a detailed definition of this,
more general, version in \refS{Sprep}, and use it in our results below.

Results on the asymptotic distribution of $X_n$ as \ntoo{} have been given
by many authors under varying assumptions, using different methods.
It is well-known that the asymptotic behaviour of $X_n$ depends on the
eigenvalues of $R$, or equivalently of its transpose $A=R^t$,
see \eg{}  \cite[Theorems 3.22--3.24]{SJ154}.
By the Perron--Frobenius theory of positive matrices
(applied to $R+cI$ for some $c\ge0$),
$R$ has a largest real eigenvalue $\gl_1$, and all other eigenvalues $\gl$
satisfy $\Re\gl<\gl_1$. 
We say that an eigenvalue $\gl$ is \emph{large} if $\Re\gl>\frac12\gl_1$,
\emph{small} if $\Re\gl\le\frac12\gl_1$ and 
\emph{strictly small} if $\Re\gl<\frac12\gl_1$.
Similarly, we say that the \Polya{} process (or urn) is \emph{small} 
(\emph{strictly small})
if $\gl_1$ is simple and all other eigenvalues are small (strictly small);
a process is large whenever it is not small.
We call a \Polya{} process \emph{critically small}
if it is small but not strictly small,
\ie, if the process is small and
$R$ admits an eigenvalue $\gl$ such that $\Re\gl=\gl_1/2$.
We define, letting $\gL$ be the set of eigenvalues,
\begin{equation}\label{gs2}
\gs_2:=
\begin{cases}
\max\bigset{\Re\gl:\gl\in\gL\setminus\set{\gl_1}},
&\text{$\gl_1$ is a simple eigenvalue;}
\\
\gl_1,&\text{$\gl_1$ is not simple.}
\end{cases}
\end{equation}
Thus the \Polya{} urn is strictly small if $\gs_2<\gl_1/2$,  
critically small if $\gs_2=\gl_1/2$, 
and large if $\gs_2>\gl_1/2$.

In the main results we assume that the urn is irreducible, 
\ie, that  the matrix $R$ is irreducible.
(In other words, every colour is dominating
in the sense of  \cite{SJ154}.)
Then, 
the largest eigenvalue $\gl_1$ is simple.
(Thus the second case in \eqref{gs2} does not occur.)
As said above, we also assume the urn to be balanced, with all row sums of
$R$ equal to $m$, and then $\gl_1=m$, with a corresponding right eigenvector
$(1,\dots,1)$.
Furthermore, there exists a positive left eigenvector $v_1$ of $R$
with eigenvalue $m$;
we assume that $v_1$ is normalized by $|v_1|=1$, and then $v_1$ is unique.

If the urn is irreducible and small, then $X_n$ is asymptotically normal
\cite[Theorems 3.22--3.23]{SJ154}.
More precisely, if $v_1$ is the positive eigenvector of $R$ defined above,
and $\nu=0$ if the urn is strictly small and $\nu\ge1$ 
is the integer defined in \refT{T=} below if the urn is critically small,
then, as \ntoo,
\begin{equation}\label{small}
\frac{X_n-n\gl_1v_1}{\sqrt{n\log^\nu n}}\dto N(0,\Sigma),  
\end{equation}
where
the asymptotic covariance matrix $\Sigma$ can be computed from $R$.
(See \eg{} \cite[Lemma 5.3 and Lemma 5.4 with (2.15) and (2.17)]{SJ154}.)
On the other hand,
by \cite[Theorems 3.24]{SJ154} and, in particular, \cite[Theorems 3.5--3.6]{P},
if the urn is large,
then there exist (complex) random
variables
$W_k$,
(complex)
left
eigenvectors $v_k$ of
$R$ 
and an integer $\nu\ge0$
such that, \as{} and in any $L^p$, 
\begin{equation}\label{large}
  X_n=n\gl_1 v_1+
\sum_{k:\Re\gl_k=\gs_2} n^{\gl_k/\gl_1}\log^\nu n\,
 W_kv_k +
  o\bigpar{n^{\gs_2/\gl_1}\log^\nu n}. 
\end{equation}
In general, there will be oscillations (coming from complex eigenvalues
$\gl_k$) and $X_n$ will not converge in distribution (after any non-trivial
normalization).
Mixed moments of the limit distributions $W_k$ in~\eqref{large} can be
computed, see~\cite{P}.
However, there is in general no explicit description of the limit laws for a
large urn. 
See~\cite{ChaPouSah},~\cite{ChaMaiPou},~\cite{ChaLiuPou} and~\citet{Mailler}
for some recent improvements on these distributions. 
Note also that~\eqref{large} is valid as soon as the urn is large and $\lambda _1$ a simple eigenvalue, the urn being irreducible or not (see~\cite{P}).

Results of this type have been proven by several authors, under
varying assumptions, using several different methods. The proofs in
\citet{SJ154} 
use an embedding in a continuous-time multi-type branching process, a method
that was introduced by \citet{AthreyaKarlin}. This method leads to general
results on convergence in distribution, but not 
to results on the moments.
A different method was developed by \citet{P}, where algebraic expressions
were obtain for (mixed) moments of various components of $X_n$, and
asymptotics were 
derived. For large urns, the resulting moment estimates and some simple
martingale arguments give the limit results, with convergence a.s.\ and in
$L^p$, and thus convergence of all moments (after suitable normalization). 
The method applies  also to small urns, and yields limits for
the moments. In principle, it should be possible to use the resulting
expressions and the method of moments to show \eqref{small}. However, the
expressions for the limits are a bit involved, and it seems difficult to do
this in general.

The purpose of the present paper is to show moment convergence for small
urns by combining these two methods.
We use the convergence in distribution \eqref{small} proven in \cite{SJ154}, and
we use the estimates of moments proven in \cite{P} to show that any moment
of the \lhs{} of \eqref{small} is bounded as \ntoo; these together imply
moment convergence in \eqref{small}. (We thus do not have to calculate the
limits 
provided by \cite{P} exactly; it suffices to find bounds of the right order
of magnitude.) This yields the following theorems, which are our main
results.

All limits and $o(\dots)$ 
in this paper are as \ntoo.

\begin{theorem}\label{T<}
Suppose that the urn is balanced, irreducible and strictly small.
Then \eqref{small} holds, with $\nu=0$, with convergence of all moments.
In particular,
$\E X_n = n\gl_1 v_1 + o\bigpar{n\qq}$
and the covariance matrix
$\Varx{X_n}=n\gS+o(n)$. 
\end{theorem}

\begin{theorem}\label{T=}
Suppose that the urn is balanced, irreducible and critically small.
Let
$1+d$
be the dimension of the largest Jordan block of $R$ corresponding to
an eigenvalue $\gl$ with $\Re\gl=\gl_1/2$ ($d\ge0$).
Then \eqref{small} holds, with $\nu=2d+1$, with convergence of all moments.
In particular,
$\E X_n = n\gl_1 v_1 + o\bigpar{(n\log^\nu n)\qq}$
and the covariance matrix
$\Varx{X_n}=\bigpar{n \log^\nu n}\gS+o(n\log^\nu n)$. 
\end{theorem}

\begin{corollary}\label{C1}
Suppose that the urn is balanced, irreducible and small, so \eqref{small} holds.
  Let $w=(w_1,\dots,w_s)$ be any vector in $\bbR^s$ and let 
$Y_n:=\wprod{X_n}=\sumis w_iX_{ni}$. 
Then
$\E Y_n = n\gl_1 \wprod{v_1} + o\bigpar{(n\log^\nu n)\qq}$ 
and
$\Var Y_n = \bigpar{\gam+o(1)} n \log^\nu n$, where
$\gam=w^t\gS w$. Moreover, if $\gam\neq0$, then
\begin{equation}\label{c1}
  \frac{Y_n-\E Y_n}{\sqrt{\Var Y_n}}\dto N(0,1)
\end{equation}
with convergence of all moments.
\end{corollary}

\begin{remark}
For the mean and variance, similar results
are also proven in
\cite{SJ307} by a related but somewhat different method (under somewhat
more general assumptions); that method does not seem to generalise easily to
higher moments.
\end{remark}

\begin{remark}\label{Rdegenerate}
  If the urn is strictly small, then it can be verified from
\cite[Lemma 5.4 and (2.13)--(2.15)]{SJ154} that $\gam=0$ 
in \refC{C1}
only in the trivial
case when $w=cu_1+u_0$ with $c\in\bbR$, $u_1=(1,\dots,1)$ and $Ru_0=0$,
which implies that $\innprod{u_0, X_n}$ is constant and thus 
$Y_n=\wprod{X_n} = Y_0 + n cm$ is deterministic, see \cite[Theorem 3.6]{SJ307}. 

On the other hand, in the critically  small case, the rank of $\gS$ is
typically only 1 or 2, and there are non-trivial vectors $w$ such that
$\gam=0$ and thus $\Var(Y_n)=o(n\log^\nu n)$.
\end{remark}

\begin{remark}
  More precise error estimates in Theorems \ref{T<} and \ref{T=}
can be obtained from the proofs below. 
In particular, for the expectation we have
in the strictly small case
$\E X_n = n\gl_1 v_1 + O\bigpar{n^{\gs_2/\gl_1}\log^{\nu_1}n}+O(1)$ 
for some $\nu_1$.
See also \cite{SJ307}.
\end{remark}

\begin{remark}
It is possible to let balls of different colours have different activities,
say $a_i\ge0$ for balls of colour $i$,
with the probability of a ball being drawn proportional to its activity
\cite{SJ154}.
The condition that the urn is balanced is now that
the total activity added each time is a constant.
In the case when all activities are positive, this is easily reduced
to the standard case $a_i=1$ by using the real version above; we just
multiply the number of balls of colour $i$ by $a_i$ (both in the urn and in
the replacement matrix). 
In general, where there are ``dummy balls'' of activity 0, which thus never are
drawn (see \eg{} \cite{SJ154} for the use of such balls), 
the results above still hold, 
assuming that the urn is irreducible if dummy balls are ignored. (Note that
we get another \Polya{} process by ignoring dummy balls, and that the
non-zero eigenvalues remain the same.)
This can be shown
by the same proofs as given below; we only have
to modify the definitions 
of balanced in \eqref{balance} and of $A$ and
$\Phi$ in \eqref{A} and \eqref{Phi} 
by replacing $\ell_k$ by $a_k\ell_k$, and note that
it is  easy to verify that the results in \cite{P} still hold
(with the corresponding modification of $\Phi_\partial$ defined there).
\end{remark}

\begin{remark}
  \label{Rnegative}
The condition $r\xij\ge0$ when $i\neq j$ 
(and \eqref{ten1}--\eqref{ten2} below)
is customary but can be relaxed
if we assume that the urn is tenable for some other reason. 
(Typically
because balls of two different colours always occur together in a fixed
proportion, and are added or subtracted together.)
See \cite[Example 7.2.(5)]{P} for a typical example and 
\cite[Remark 6.3]{SJ309} for another.
As remarked in \cite[page 295]{P}, the results in \cite{P} that we use hold
in this case too, and it follows that all moment estimates in the present
paper hold. Also \eqref{small} holds, at least under some supplementary
assumptions, see \cite[Remark 4.2]{SJ154}, and then the results above hold.
(In the examples from \cite{P} and \cite{SJ309} just mentioned,
\eqref{small} holds because there is an equivalent urn with random
replacements that satisfies the conditions of \cite{SJ154}.)
\end{remark}

\begin{remark}
  It is possible to let the replacement vectors $(r\xij)_{j=1}^s$ be random,
  see \cite{SJ154}:
with our notations of Section~\ref{Sprep}, assume that random $V$-valued
increment vectors $W_1,\dots ,W_s$ are given and that they admit moments of
order $p$, $p\geq 2$ being an integer or $\infty$. 
In this case, the conditional transition probabilities~\eqref{condProb} keep
the same form, and $X_{n+1}=X_n+W_K^{(n)}$, where, 
given $K=k$,
$W_K^{(n)}$ is a copy of $W_k$, independent of everything that has happened
so far. 
The
tenability assumptions~\eqref{ten1}--\eqref{ten2} must be 
modified: it is sufficient that $l_j(W_k)\ge0$ \as{} for all $j,k$;
more generally \eqref{ten1} should hold \as, while for each $k$ either
$l_k(W_k)\ge0$ \as{} or there exists $d_k>0$ such that \as{}
$l_k(W_k)\in\set{-d,0,d,\dots}$ while $l_k(X_0),l_k(W_i)\in\set{d,2d,\dots}$
for $j\neq k$.
Assume further that the urn is \emph{almost surely balanced}, which means
that~\eqref{balance} is a.s.\ satisfied (replacing $w_k$ by $W_k$). 

\emph{Then, our results extend to this case, the moment convergence being
  valid up to order $p$.} 

To see this, note first that
in this random replacement context, all results of~\cite{SJ154} hold.
The techniques developed in~\cite{P} and the arguments given in the present paper remain also valid after the following adaptations:
the replacement operator~\eqref{A} is now
\begin{equation}
A(v):=\sumks l_k(v)\E W_k  
\end{equation}
while the transition operator~\eqref{Phi}, restricted to polynomials $f$ of degree not more than $p$, becomes
\begin{equation}
\Phi(f)(v):=\sumks l_k(v) \E\bigpar{f(v+W_k)-f(v)}.  
\end{equation}
\end{remark}

\begin{remark}
For an example of applications of the results above on random tree processes ($m$-ary search trees and preferential attachment trees), one can refer to~\cite[Remark 3.3]{SJ309}.
\end{remark}

\begin{problem}
  As said above, we consider in this paper only balanced urns. It is a
  challenging open problem to extend the results to non-balanced urns.
\end{problem}

\section{Preliminaries}\label{Sprep}

We follow \cite{P} and use the following coordinate-free description of the
urn process.
It is easily seen to be equivalent to the traditional
description in \refS{S:intro}, with $r\xij=l_j(w_i)$ and allowing these numbers
to be real and not necessarily integers.

Let $V$ be a real vector space of finite dimension $s\ge1$
and let $l_1,\dots,l_s$ be a basis of the dual space $V'$;
let $V_+:=\set{v\in V:l_j(v)\ge0, j=1,\dots,s}\setminus\set0$ be the
positive orthant. Let $X_0$ and $w_1,\dots,w_s$ be given vectors in $V$,
with $X_0\in V_+$.

Given $X_n\in V_+$, for some $n\ge0$, we let $X_{n+1}:=X_n+w_K$, where the 
random index $K$ is chosen with conditional probability, given $X_n$,
\begin{equation}
\label{condProb}
  \P(K=k\mid X_n)=\frac{\ell_k(X_n)}{\sumjs\ell_j(X_n)}.
\end{equation}
This defines the \Polya{} process $(X_n)_0^\infty$ (as a Markov process), 
provided the process is
\emph{tenable}, \ie, $X_n\in V_+$ for all $n$.

The standard sufficient set of conditions for tenability, used by many
authors, is in our formulation:
for all $j,k=1,\dots,s$,
\begin{align}
    l_j(w_k)&\ge0 \quad\text{if }j\neq k,\label{ten1} \\ 	
l_k(w_k)&\ge0 \quad \text{or}\quad l_k(X_0)\bbZ+\sumis l_k(w_i)\bbZ 
= l_k(w_k)\bbZ.
\label{ten2}
\end{align}
We assume \eqref{ten1}--\eqref{ten2} for simplicity, but as said in
\refR{Rnegative}, the results hold more generally under suitable conditions.

In the present paper, we also assume that the process is \emph{balanced},
which in this context means
\begin{equation}\label{balance}
  \sumks l_k(w_j)=m,
\qquad j=1,\dots,s,
\end{equation}
for some fixed $m$. We assume further $m>0$, and we may  without
loss of generality assume $m=1$, since we may divide all $X_n$ and $w_k$
(or, alternatively, all $l_j$) by $m$.

We shall also use the following notation from \cite{P}, where further
details are given.

The replacement matrix $R$ (or rather its transpose)
now corresponds to the \emph{replacement operator}
$A:V\to V$ defined by
\begin{equation}
  \label{A}
A(v):=\sumks l_k(v)w_k.
\end{equation}

We choose a basis $(v_k)_1^s$ in the complexification $V_\bbC$ that yields a
Jordan block decomposition of $A$, and let $(u_k)_1^s$ be the corresponding
dual basis in $V'_\bbC$. 
We may assume that these vectors are numbered such that $u_1$ and $v_1$
correspond to the eigenvalue $\gl_1=m=1$,
and, moreover, for each $k$ either
$u_k\circ A=\lambda _ku_k$ 
(so $u_k$ is an eigenvector of the dual operator $A'$) or
$u_k\circ A=\lambda _ku_k+u_{k-1}$, for some eigenvalue $\gl_k$.
Since the urn is supposed to be irreducible, $\gl_1=1$ is a simple eigenvalue;
furthermore, the balance condition
\eqref{balance} (with $m=1$) implies that $\sumjs l_j\in V'$ is an
eigenvector of  $A'$ with eigenvalue $1$; hence we may
assume that $u_1=\sumjs l_j$.
This means that $v_1$ is normalized by $\sumjs l_j(v_1)=1$.

Let $\gl:=(\gl_1,\dots,\gl_s)$, the vector of eigenvalues of $A$ (repeated
according to algebraic multiplicity).

Let $\pi_k$ denote the projection of $V_\bbC$ onto $\bbC v_k$ defined by
$\pi_k(v):= u_k(v)v_k$. Note that $\sumks \pi_k=I$. 

For a multi-index $\ga=(\ga_1,\dots,\ga_s)\in\bbZss$, let
$\uu^\ga:=\prod_{i=1}^s u_i^{\ga_i}$;
this is a homogeneous polynomial function on
$V^s_\bbC$.
We call such multi-indices $\ga$ \emph{powers}, and we say that
$\ga$ is a \emph{small power} if only linear forms $u_i$ corresponding to
small eigenvalues appear in $\uu^\ga$, \ie, if $\Re\gl_i\le\frac12$ when
$\ga_i>0$;
we define \emph{strictly small power} in the same way.

Let $\Phi$ be the linear operator in the space of (complex-valued) functions on $V$ defined by
\begin{equation}\label{Phi}
  \Phi(f)(v):=\sumks l_k(v) \bigpar{f(v+w_k)-f(v)}.
\end{equation}

We order the multi-indices by the \emph{degree-antialphabetic order}, see
\cite{P}, and
define $S_\ga:=\spann\set{\uu^\gb:\gb\le\ga}$.
Then $S_\ga$ is a finite-dimensional space of polynomials, and $S_\ga$ is
$\Phi$-stable \cite[Proposition 3.1]{P}.
Thus $S_\ga$ has a decomposition into generalized eigenspaces 
$\ker(\Phi-z)^\infty:=\bigcup_n \ker(\Phi-z)^n$, and we define 
the \emph{reduced polynomial} $Q_\ga$ as
the projection of $\uu^\ga$ onto $\ker(\Phi-\innprodgl{\ga})^\infty$ in
this decomposition. Then, for any $\ga\in\bbZss$, $\set{Q_\gb:\gb\le\ga}$ is a
basis in $S_\ga$ \cite[Proposition 4.8(2)]{P}.
Furthermore, the following statement follows from the more precise
\cite[Proposition 5.1]{P}.

When $\ga$ is any power, we denote by $\nu_\ga$ the index of nilpotence of
$Q_\ga$ for $\Phi -\innprodgl{\ga}$, defined by
\begin{equation}\label{index}
1+\nu _\ga=
\min\bigset{ p\ge1:\bigpar{\Phi-\innprodgl{\ga}}^p\bigpar{Q_\ga}=0}.  
\end{equation}
Since $Q_\ga$ belongs to the generalized eigenspace space
$\ker\bigpar{\Phi-\innprodgl{\ga}}^\infty$, this index is finite. 
In particular, $\nu _\ga=0$ if and only if $Q_\ga$ is an eigenfunction
of~$\Phi$.

\begin{proposition}\label{momQ}
For any $\ga\in\bbZss$, 
\begin{equation}\label{qc10}
  \E Q_\ga(X_n) = O\bigpar{n^{\Re\innprodgl{\ga}}\log^{\nu_\ga}n},
\end{equation}
where $\nu_\ga$ is the index of nilpotence of $Q_\ga$ defined in \eqref{index}.
\end{proposition}

Our proofs use the whole machinery of~\cite{P}.
We define a polyhedral cone  $\gS$ and, for every power $\ga$, a polyhedron
$A_\ga$  (to be precise, the set of integer points in a convex polyhedron). 
Let $\gd_j$ denote the multi-index $\ga$ with $\ga_i=\gd_{ij}$,
\ie, a single 1 in the $j$-th place.
The cone $\gS$ is can be defined by its spanning edges:
\begin{equation}
\label{edgesSigma}
\gS
:=\sum_ {(i,j)\in\set{1,\dots s}^2,~i\neq j} \bbR_{\ge 0}\left( 2\gd_i-\gd_j\right)
\end{equation}
or equivalently as an intersection of half-spaces:
\begin{equation}
\label{facesSigma}
\gS
:=
\bigcap_{\substack{I\subseteq\set{1,\dots ,s}
}}
\{x\in\bbR^s:\gd_I^*(x)\ge 0\}
\end{equation}
where 
\begin{equation}\label{gdIx}
\gd_I^*(x_1,\dots ,x_s)=\sum_{1\le i\le s}x_i+\sum _{i\in I}x_i  
\end{equation}
for every subset $I$ of $\{ 1,\dots ,s\}$;
the equivalence between the two definitions is proven in \cite{P}.
(Moreover, it suffices to consider $I$ with $1\le\#I\le s-1$ in
\eqref{facesSigma}; these $I$ correspond to the faces of $\gS$, see \cite{P}.)

When $\ga\in\bbZss$, the polyhedron $A_\ga$ is defined as
\begin{equation}\label{aga}
A_\ga =(\ga -D_\ga )\cap\bbZss  
\end{equation}
where 
$\ga -D_\ga$ denotes $\{\ga -d:d\in D_\ga\}$ and
$D_\ga$ is 
\xfootnote{The definition of $D_\ga$ corrects a minor error in \cite{P}.}
the set of $\bbZ_{\geq 0}$-linear combinations
of all vectors $\gd _k-\gd _{k-1}$ such that
$u_k$ is not an eigenfunction of $A'$.
Note that for such $k$, $\gl_{k-1}=\gl_k$; hence, if $\ga'\in A_\ga$, then
\begin{equation}
  \label{agas}
\sum_{k:\gl_k=z}\ga'_k=\sum_{k:\gl_k=z}\ga_k 
\qquad \text{for every $z\in\bbC$;}
\end{equation}
as a consequence, 
$|\ga'|=|\ga|$ and
$\innprodgl{\ga' }=\innprodgl{\ga }$.
Note also that always $\ga\in A_\ga$, 
and that if $A$ is diagonalizable, then $D_\ga=\set0$, and thus
$A_\ga =\set\ga$.

We use the following theorem,
proven in \cite{P}.
It describes more precisely the action of $\Phi$ on the generalized
eigenspace $\ker\left(\Phi -\langle\lambda ,\alpha\rangle\right)^\infty$,
which has $\set{Q_\gb :\innprodgl{\gb}=\innprodgl{\ga}}$ as a basis. 
$A_\ga -\gS$ denotes $\set{\ga'-\gs:\ga'\in A_\ga,~\gs\in\gS}$.

\begin{theorem}[{\cite[Proposition 4.8(5) and Theorem 4.20]{P}}]
\label{stabilityFalpha}
Let $\ga\in\bbZss$.
\begin{romenumerate}
\item\label{sFa} 
$\bigpar{\Phi -\innprodgl{\ga}}\bigpar{Q_\ga}\in \spann\set{Q_\gb:\gb<\ga,\,
\innprodgl{\gb}=\innprodgl{\ga}}$.
\item \label{sFb}
The subspace
\begin{equation}
\label{Fga}
\xF_\ga:=\spann\set{\uu^\gb:\gb\in (A_\ga-\gS)\cap \bbZgeo^s}
\end{equation}
is $\Phi$-stable, and
\begin{equation}
\label{FgaQ}
\xF_\ga=\spann\set{Q_\gb:\gb\in (A_\ga-\gS)\cap \bbZgeo^s}.
\end{equation}
In particular,
$\bigpar{\Phi -\innprodgl{\ga}}\bigpar{Q_\ga}\in \xF_\ga$.
\item \label{sFc}
As a consequence,
\begin{equation}\label{Fgax}
  \bigpar{\Phi -\innprodgl{\ga}}\bigpar{Q_\ga}
\in 
\spann\bigset{Q_\gb:\gb\in K_\ga},
\end{equation}
where
\begin{equation}\label{Kga}
K_\ga:=\bigset{\gb\in (A_\ga-\gS)\cap \bbZgeo^s:\gb<\ga,
\,\innprodgl{\gb}=\innprodgl{\ga}}.
\end{equation}
\end{romenumerate}
\end{theorem}

\section{Proofs}

Recall that we for convenience, and without loss of generality, assume
$\gl_1=m=1$.

\subsection{Powers and nilpotence indices}
We begin with the strictly small case, which is rather simple.

\begin{lemma}\label{L1}
  If $\ga$ is a strictly small power, then $\Re\innprodgl{\gb}\le|\ga|/2$
  for any $\gb\in\bbZss\cap(A_\ga-\gS)$, with equality only if $\gb=c\gd_1$
  with $c=|\ga|/2$.
\end{lemma}

\begin{proof}
Let $\ga'\in A_\ga$ and $\gs\in\gS$ such that $\gb=\ga'-\gs$.
Also, let
$I:=\set{k:\Re\gl_k\ge\frac12}$
and recall \eqref{gdIx}. 
Since each $\gb_k\ge0$
and each $\Re\gl_k\le1$,
\begin{align}\label{b1}
  \Re\innprodgl{\gb}
=\sum_k \gb_k\Re\gl_k
&
\le
\sum_{k:\Re\gl_k<\frac12}\tfrac12\gb_k
+ \sum_{k:\Re\gl_k \ge\frac12}\gb_k
\\&
=
\tfrac12 \gdIx(\gb)
=
\tfrac12 \gdIx(\ga')
-
\tfrac12 \gdIx(\gs).  
\end{align}
 Since $\ga$ is a strictly small power, 
\eqref{agas} implies that $\ga'\in A_\ga$ also is a strictly small power
and that $\gdIx(\ga')=|\ga'|=|\ga|$.
Furthermore, 
the definition \eqref{facesSigma}
of $\gS$ by its faces guarantees that $\gd_I^*(\gs)\ge 0$. 
Hence, $\Re\innprodgl{\gb}\le\frac12|\ga|$.

Finally, suppose that equality holds. This implies equality in
\eqref{b1}, which can hold only if $\gb_k=0$ when $\Re\gl_k\neq1$, which means that
$\gb=c\gd_1$ with 
$c=\gb_1$. Furthermore, then
$|\ga|/2=\innprodgl{\gb}=c\innprodgl{\gd_1}=c\gl_1=c$.
\end{proof} 

The rest of this subsection is devoted to the critically small case, where
we have to pay special attention to eigenvalues $\gl$ with $\Re\gl=\frac12$;
such eigenvalues are called \emph{critical}.
Recall that we have chosen a basis $(v_1,\dots,v_s)$ that yields a Jordan
block decomposition of $A$. 
A set of indices $J\subseteq\set{1,\dots ,s}$
that corresponds to a Jordan block 
is called a \emph{monogenic block of indices} \cite{P}; if the corresponding
eigenvalue is critical, $J$ is called a \emph{critical monogenic block}.

The \emph{support} of a  power 
or another vector $\ga=(\alpha _1,\dots ,\alpha _s)\in\bbZ^s$
is $\supp(\ga):=\set{k:\ga_k\neq0}$.
The power (vector) $\ga$ is
called \emph{critical} if
$\ga_k\neq0\implies\Re\gl_k\in\set{1,\frac 12}$,
and $\ga$ is called \emph{strictly critical} if
$\ga_k\neq0\implies\Re\gl_k=\frac 12$.
Furthermore, $\ga$ is called
\emph{monogenic} when
its support in contained in some monogenic block $J$,
and
$\ga$ is called a 
\emph{quasi-monogenic  power} when
$\supp(\ga)\subseteq\set{1}\cup J$
for some monogenic block $J$.
We consider only critical monogenic blocks, \ie, blocks associated to a critical
eigenvalue.
(Note that a power $\ga=c\gd_1$ is critical and quasi-monogenic, and
associated to any monogenic block $J$; otherwise $J$ is determined by $\ga$.)

Recall that $K_\ga$ is the set of powers defined in \eqref{Kga}.

\begin{lemma}\label{criticalCase}
Assume that the urn is critically small.
\begin{romenumerate}
\item \label{cca}
Let $\ga$ be a critical power and let
$\gb\in (A_\ga-\gS)\cap \bbZgeo^s$. Then,
$\Re\innprodgl{\gb}\le\Re\innprodgl{\ga}$, with equality only if $\gb$ is
critical. 

\item \label{cck}
If $\ga$ is a critical power, then any $\gb\in K_\ga$ is critical.

\end{romenumerate}
\end{lemma}

\begin{proof}
\pfitemref{cca}
Let $\gb:=\ga'-\gs$ with $\ga'\in A_\ga$ and $\gs\in\gS$.
Then
\begin{equation}
  \label{qp}
\innprodgl{\gb}
=\innprodgl{\ga'}-\innprodgl{\gs}
=\innprodgl{\ga}-\innprodgl{\gs}.
\end{equation}
Furthermore, since $\ga$ is critical, it follows from \eqref{agas} that
$\ga'$ too is critical.
Hence for an index $k$ with $\Re\gl_k<\frac12$, we have $\ga'_k=0$ and thus 
$\gb_k=-\gs_k$ so $\gs_k\le0$.
Since the urn is critically small, it follows that
\begin{equation}
  \begin{split}
\Re\innprodgl{\gs}
&=\gs_1+\sum_{k:\Re\gl_k<\frac12}\gs_k\Re\gl_k	
+\sum_{k:\Re\gl_k=\frac12}\tfrac12\gs_k	
\\&
\ge\gs_1+\sum_{k:\Re\gl_k<\frac12}\tfrac12\gs_k
+\sum_{k:\Re\gl_k=\frac12}\tfrac12\gs_k	
=\frac12\gd_{\set1}^*(\gs)\ge0,
  \end{split}
\end{equation}
where the last inequality comes from \eqref{facesSigma}.
Hence, \eqref{qp} yields
$\Re\innprodgl{\gb}\le\Re\innprodgl{\ga}$;
moreover, equality holds only if 
$\Re\gl_k<\frac12$ implies $\gs_k=0$ and thus $\gb_k=\ga'_k=0$,
\ie, $\gb$ is critical. (Equality also requires 
$\gd^*_{\set{1}}(\gs)=0$.)

\pfitemref{cck}
Let $\ga$ be a critical power.
If $\gb\in K_\ga$, then $\beta\in (A_\ga-\gS)\cap \bbZgeo^s$ and equality holds in~\ref{cca}.
Then $\gb$ is critical.
\end{proof}

As a consequence of Lemma~\ref{criticalCase} and \refT{stabilityFalpha},
the space $\mathcal C$ of polynomial functions on $V$ defined by
\begin{equation}
\mathcal C:=\spann\bigset{ Q_\ga:\ga~{\rm critical}}
\end{equation}
is $\Phi$-stable; thus,
when $\ga$ is a critical power,
$\nu_\ga$ is also the index of $Q_\ga$ for the nilpotent endomorphism
induced by $\Phi-\innprodgl{\ga}$ on $\mathcal C$.
This property is the basic fact that allows us to prove
Proposition~\ref{nuleq} which constitutes the key argument of
Theorem~\ref{T=}. 

\begin{proposition}\label{nuleq}
Assume that the urn is critically small.
If $\ga$ is a quasi-monogenic critical power associated with a Jordan
block of size $1+r$, $r\ge0$, then $\nu_\ga\le\bigpar{r+\frac 12}|\ga|$. 
%
\end{proposition}

The remainder of this section is devoted to the proof of
Proposition~\ref{nuleq}. 
We assume that $\ga$ is a critical power with 
$\supp(\ga)\subseteq\set1\cup J$ for some monogenic block $J$, and we may
without loss of generality assume that $J=\set{2,\dots,r+2}$ for some $r\ge
0$, since we otherwise may permute the Jordan blocks of the chosen basis. 
In this case, we define for vectors $\gam$ with
$\supp(\gam)\subseteq\set1\cup J$,
\begin{equation}\label{M}
M(\gam):=
\sum_{k=1}^{r+2} k\gam_k-2\sum_{k=1}^{r+2}\gam_k+\Re\innprodgl{\gam}
=\sum_{k=2}^{r+2}\bigpar{k-\tfrac 32}\gam_k.
\end{equation}
Note that $M(\gam)$ is a linear function of $\gam$.

\begin{lemma}\label{MAalpha}
Assume that $\ga$ is a quasi-monogenic critical power 
with  monogenic block $J=\set{2,\dots,r+2}$,
$r\ge 0$. 
Let $\ga'\in A_\ga\setminus\set{\ga}$.
Then, $\ga'$ is also a critical quasi-monogenic power
with monogenic block $J$ and $M(\ga)\le M(\ga')-1$.
\end{lemma}

\begin{proof}
By~\eqref{aga} and~\eqref{agas}, only the inequality is non-trivial.
Furthermore, $\ga'$ can be written
$$
\ga'=\ga-\sum _{3\le k\le r+2}\ep_k\bigpar{\gd_k-\gd_{k-1}}
$$
where the $\ep_k$ are nonnegative integers.
Then, since $M\bigpar{\gd_k-\gd_{k-1}}=1$,
\begin{equation}
M(\ga')=M(\ga)-\sum _k\ep_kM\bigpar{\gd_k-\gd_{k-1}}
=M(\ga)-\sum _k\ep_k
<M(\ga).
\end{equation}
\end{proof}

\begin{lemma}\label{Malpha-Sigma}
Assume that the urn is critically small.
Let $\ga$ be a quasi-monogenic critical power 
with  monogenic block $J=\set{2,\dots,r+2}$,
$r\ge 0$. 
Assume that $\gb\in\bigpar{\ga-\gS}\cap\bbZgeo^s$ satisfies\/
$\Re\innprodgl{\gb}=\Re\innprodgl{\ga}$
and\/ $\gb\neq\ga$.
Then, $\gb$ is also a critical quasi-monogenic power with monogenic block  
$J$ and $M(\gb)\le M(\ga)-1$.
\end{lemma}

\begin{proof}
When $i,j\in\set{1,\dots ,s}$ are distinct, denote by $\gd_{(i,j)}$ the $s$-dimensional vector $\gd_{(i,j)}=2\gd_i-\gd_j$.
These vectors span $\gS$, see~\eqref{edgesSigma}.
We divide the proof into three steps.

\smallskip
\textcircled{\tiny 1} 
\emph{Let  $i,j$ be distinct indices in \set{1,\dots,s}.
Then $\gdix(\gd_{(i,j)})\ge0$ with equality if and only if $j=1$.
}

Indeed, by \eqref{gdIx},
$\gdix(\gd_{(i,j)})=2+2\gd_{i1}-1-\gd_{j1}$ and the result follows. 



\smallskip
\textcircled{\tiny 2} 
\emph{Let $\gs=\ga-\gb\in\gS$.
Then, $\gs$ is a linear combination of $\gd_{(k,1)}$, $k\in J$, with
nonnegative coefficients.} 

Indeed, Lemma~\ref{criticalCase} guarantees that $\gb$ is critical, so that
$\gs$ is also critical.
Consequently, by \eqref{gdIx},
$\gdix(\gs)=2\Re\innprodgl{\gs}$.
Furthermore, by the assumption,
$\Re\innprodgl{\gs}=\Re\innprodgl{\ga}-\Re\innprodgl{\gb}=0$.
Hence, $\gdix(\gs)=0$.

Since $\gs$ is a linear combination of vectors $\gd_{(i,j)}$ with nonnegative
coefficients (definition~\eqref{edgesSigma} of $\gS$ by edges),
\textcircled{\tiny 1} proves that all $j$ that appear are equal to $1$.
Thus
\begin{equation}\label{gsgd}
\gs=\sum_{k=2}^s\ep_k\gd_{(k,1)}  
\end{equation}
where the $\ep_k$ are
nonnegative (real) numbers.
Furthermore, if $k\ge2$ and $k\notin J$, then
$
  0=\ga_k\ge\ga_k-\gb_k=\gs_k=2\eps_k\ge0
$ 
and thus $\eps_k=0$.

\smallskip
\textcircled{\tiny 3}
It follows from \textcircled{\tiny 2} that  $\supp(\gs)\subseteq\set{1}\cup J$,
and thus
this is also true for $\gb$, proving the assertion that $\gb$ is critical
and quasi-monogenic with monogenic block $J$.
Furthermore, by \eqref{gsgd} and \eqref{M},
\begin{equation}
  M(\gs)=\sum_{k=2}^{r+2}\eps_kM(\gd_{(k,1)})
=\sum_{k=2}^{r+2}\eps_k \xpar{2k-3}
\ge
\sum_{k=2}^{r+2}\eps_k =-\gs_1\ge1
\end{equation}
since $\gs_1$ is an integer and the sum is nonnegative and nonzero (because
$\gb\neq\ga$). 
Consequently, $M(\gb)=M(\ga)-M(\gs)\le M(\ga)-1$.
\end{proof}

\begin{lemma}\label{nuleqM}
Assume that the urn is critically small.
Let $\ga$ be a quasi-monogenic critical power with  monogenic block  
$\{ 2,\dots,r+2\}$, $r\ge0$.
Then $\nu_\ga\le M(\ga)$.
\end{lemma}

\begin{proof}
Let $J=\set{2,\dots,r+2}$ be a critical monogenic block and
fix $\ell\in\frac12\bbZgeo$. 
Let 
\begin{equation}
I_\ell:=\bigset
{\ga\in\bbZss:\supp (\ga)\subseteq\set{1}\cup J,~\Re\innprodgl{\ga}=\ell}.  
\end{equation}
We show by induction on $\ga$ 
(using the degree-antialphabetical order)
that the inequality $\nu_\ga\le M(\ga)$ is
true for every $\ga\in I_\ell$.
Note that $I_\ell$ is finite and thus well-ordered.

Take any $\ga\in I_\ell$ and suppose by induction that $\nu _\gb\le M(\gb)$
for any $\gb\in I_\ell$ such that $\gb<\ga$. 
By \refT{stabilityFalpha}, \eqref{Fgax}--\eqref{Kga} hold.
In particular, by the definition of the index of nilpotence, 
\begin{equation}
  \label{nuind}
\nu _\ga\le
\begin{cases}
0, & K_\ga=\emptyset,
\\
 1+\max\set{\nu _\gb:\gb\in K_\ga},
& K_\ga\neq\emptyset.
\end{cases}
\end{equation}
In particular, if $K_\ga=\emptyset$, then $\nu_\ga=0\le M(\ga)$.

Assume $K_\ga\neq\emptyset$ and
let $\gb\in K_\ga$. Then $\gb=\ga'-\gs$ with $\ga'\in A_\ga$ and $\gs\in\gS$.
By Lemmas~\ref{MAalpha} and~\ref{Malpha-Sigma}, $\ga'$ and $\gb$ are also
critical quasi-monogenic powers with monogenic block $J$. Thus $\gb\in I_\ell$.
Furthermore, if $\ga'\neq \ga$, then Lemmas~\ref{MAalpha}
and~\ref{Malpha-Sigma} also yield $M(\gb)\le M(\ga')\le M(\ga)-1$,
while 
if $\ga'= \ga$, then
Lemma~\ref{Malpha-Sigma} yields 
$M(\gb)\le  M(\ga)-1$. Hence, in any case, $M(\gb)\le M(\ga)-1$.
By the inductive assumption, we thus have $\nu_\gb\le M(\gb)\le M(\ga)-1$.

Consequently, \eqref{nuind} shows that if $K_\ga\neq0$, then 
$\nu_\ga\le 1+(M(\ga)-1)=M(\ga)$, which completes the induction.
\end{proof}

\begin{remark}
Since $\nu_\ga$ is an integer, in fact, $\nu_\ga\le\lfloor M(\ga)\rfloor$.
Strict inequality is possible. For example, if $\gl_2=\frac12+\ii t$ 
is a critical eigenvalue with $t\neq0$, then $Q_{2\gd_2}$ is an
eigenfunction of $\Phi$ and thus $\nu_{2\gd_2}=0$.
\end{remark}

\begin{proof}[Proof of Proposition~\ref{nuleq}]
Let $J$ be a Jordan  block of size $1+r$ associated to $\ga$.
As said above, we may assume that $J=\set{2,\dots ,r+2}$.
Then, by Lemma~\ref{nuleqM} and \eqref{M},
\begin{equation}
\nu_\ga \le M(\ga)=\sum _{k=2}^{r+2}\bigpar{k-\tfrac 32}\ga_k
\le \bigpar{r+\tfrac 12}|\ga|.
\end{equation}
\end{proof}

\begin{remark}
The upper bound in Proposition~\ref{nuleq} is reached only for
$\ga=|\ga|\gd_{\max J}$ where $J$ is a critical Jordan block. 
Moreover, it is reached  only when $|\ga|$ is even, explaining why the odd
moments of $X_n$ are asymptotically negligible after normalization. 
\end{remark}

\subsection{Moments}

\begin{lemma}
\label{momSmal}
  If $\ga$ is a strictly small power, then 
$\E \uu^\ga(X_n) = O\bigpar{n^{|\ga|/2}}$. 
\end{lemma}

\begin{proof}
  Since $\uu^\ga\in \xF_\ga$ by the definition \eqref{Fga}, it follows from
  \eqref{FgaQ} that we have a decomposition
  \begin{equation}
	\label{uQ}
\uu^\ga=\sum_{\gb\in \bbZgeo^s\cap(A_\ga-\gS)}q_{\ga,\gb} Q_\gb
  \end{equation}
for some constants $q_{\ga,\gb}$.

If $\gb\in \bbZgeo^s\cap(A_\ga-\gS)$ and $\gb\neq (|\ga|/2)\gd_1$, then 
$\Re\innprodgl{\gb}<|\ga|/2$
by \refL{L1}. Furthermore, by \cite[Proposition 5.1]{P},
for some $\nu_\gb\ge0$,
\begin{equation}\label{qc1}
  \E Q_\gb(X_n) = O\bigpar{n^{\Re\innprodgl{\gb}}\log^{\nu_\gb}n}
=o\bigpar{n^{|\ga|/2}}.
\end{equation}

On the other hand, if $\gb= (|\ga|/2)\gd_1$ (and thus $|\ga|$ is even), then 
$Q_\gb$ is an eigenfunction of $\Phi$ and by \cite[Proposition 5.1(1)]{P},
\eqref{qc1} holds with $\nu_\gb=0$, so 
\begin{equation}\label{qc2}
  \E Q_\gb(X_n) = O\bigpar{n^{\innprodgl{\gb}}}
=O\bigpar{n^{|\ga|/2}}.
\end{equation}
In fact, in this case $Q_\gb=u_1(u_1+1)\dotsm(u_1+|\ga|/2-1)$ so
$Q_b(X_n)$ 
is deterministic, and a polynomial in $n$ of degree $|\ga|/2$, see
\cite[Remark 4.10]{P}.
\end{proof}

\begin{lemma}\label{momCrit}
Assume that the urn is critically small.
Let, as in \refT{T=}, 
$1+d$ be the largest dimension of a critical Jordan block of the replacement
matrix $R$. 
Then, if $\ga$ is a strictly critical power $\ga$,
\begin{equation}
  \E\uu^\ga(X_n) = O\bigpar{n\log^{2d+1}n}^{|\ga|/2}.
\end{equation}
\end{lemma}

\begin{proof}
Decomposing $\uu^\ga=\uu^{\ga_1}\dots\uu^{\ga_t}$ where the $\ga_k$ are
monogenic critical powers, thanks to the Cauchy--Schwarz inequality
applied $t-1$ times, it suffices to show the lemma when $\ga$ is strictly
critical and
monogenic. 

Suppose thus that $\ga$ is strictly critical and monogenic.
Note that, since $\ga$ is strictly critical, $\Re\innprodgl{\ga}=|\ga|/2$.
As above, we use the decomposition \eqref{uQ} of $\uu^\ga$; we now split it
as
\begin{equation}\label{uu2}
\uu^\ga=\sum_{\gb\in A_\ga-\gS,\Re\innprodgl{\gb}=\Re\innprodgl{\ga}}q_{\ga,\gb} Q_\gb
+\sum_{\beta:\Re\innprodgl{\gb}<\Re\innprodgl{\ga}}q_{\ga,\gb} Q_\gb.  
\end{equation}
When $\Re\innprodgl{\gb}<\Re\innprodgl{\ga}$, Proposition~\ref{momQ} yields
$\E Q_\gb(X_n)=o\bigpar{n^{|\ga|/2}}$. 
To deal with the first sum in \eqref{uu2}, suppose that $\gb\in A_\ga-\gS$
satisfies $\Re\innprodgl{\gb}=\Re\innprodgl{\ga}$. 
Then, thanks to Lemmas~\ref{MAalpha} and~\ref{Malpha-Sigma}, $\gb$ is also
critical and quasi-monogenic so that Proposition~\ref{nuleq} asserts that
$\nu_\gb\leq(d+\frac 12)|\ga|$. 
Thus Proposition~\ref{momQ} yields
\begin{equation}
\E Q_\beta(X_n)=O\bigpar{n^{\Re\innprodgl{\gb}}\log^{(d+\frac 12)|\ga|}n}
=O\bigpar{n^{\frac 12|\ga|}\log^{(d+\frac 12)|\ga|}n}.  
\end{equation}
Putting the small $o$ and the big $O$ together, one gets the result.
\end{proof}

\subsection{Proofs of Theorems~\ref{T<} and~\ref{T=}, and of Corollary~\ref{C1}}

\begin{proof}[Proof of Theorems~\ref{T<} and~\ref{T=}]
Assume that the urn is small.
Let $P_I:=\sum_{k:\Re\gl_k<\frac12}\pi_k$ and 
$P_{II}:=\sum_{k:\Re\gl_k=\frac12}\pi_k$, so that 
${\id}_{\bbC^s}=\pi_1+P_I+P_{II}$. 
Remember that $\pi_k(v)=u_k(v)v_k$.

$\bullet$
We first deal with $P_I$.
Let $J_I:=\set{k:\Re\gl_k<\frac12}$. Then, for
any $v\in\bbC^s$,
\begin{equation}\label{PI}
  \begin{split}
  \bigabs{P_I(v)}^2
&= \biggabs{\sum_{k\in J_I} u_k(v) v_k}^2
=\sum_{k,j\in J_I}\innprod{v_k,v_j} u_k(v)\overline{u_j(v)}.	
  \end{split}
\end{equation}
Taking the $\ell$-th power and expanding, we see that for any $\ell\ge1$, there 
exists a set of strictly small powers $\gb$ with $|\gb|=2\ell$, and constants
$c_\gb$, such that, for all $v$,
\begin{equation}\label{ineqPI}
  \bigabs{P_I(v)}^{2\ell}
  = \sum_{\gb} c_\gb \uu^\gb(v).
\end{equation}
Hence, \refL{momSmal} yields
\begin{equation}\label{cqI}
\E  \bigabs{P_I(X_n)}^{2\ell}
  = \sum_{\gb} c_\gb \E \uu^\gb(X_n)
=O\bigpar{n^\ell}.
\end{equation}

$\bullet$
For $P_{II}$, we argue as in \eqref{PI} and obtain 
an identity similar to~\eqref{ineqPI}, now for 
for a set of strictly critical powers $\gb$ with $|\gb|=2\ell$.
Hence, \refL{momCrit} yields
\begin{equation}\label{cqII}
\E  \bigabs{P_{II}(X_n)}^{2\ell}
= \sum_{\gb} c'_\gb \E \uu^\gb(X_n)
=O\bigpar{n\log^{2d+1}n}^\ell.
\end{equation}

$\bullet$
Finally, because of the balance assumption~\eqref{balance} (with $m=1$),
$\pi_1(X_n)$ is nonrandom and 
\begin{equation}\label{pi1}
\pi_1(X_n)=u_1(X_n)v_1=\bigpar{u_1(X_0)+n}v_1=nv_1+O(1).  
\end{equation}

When the urn is strictly small (\refT{T<}), 
$P_{II}=0$ and thus
\begin{equation}
X_n=\pi_1(X_n)+P_I(X_n)=nv_1+P_I(X_n)+O(1),   
\end{equation}
and \eqref{cqI} implies
\begin{equation}
  \E|X_n-nv_1|^{2\ell} =O\bigpar{n^\ell}.
\end{equation}
When the urn is critically small (\refT{T=}), we instead have
\begin{equation}
X_n=\pi_1(X_n)+P_I(X_n)+P_{II}(X_n)=nv_1+P_I(X_n)+P_{II}(X_n)+O(1),  
\end{equation}
so that~\eqref{cqI} and~\eqref{cqII} imply
\begin{equation}
  \E|X_n-nv_1|^{2\ell} =O\bigpar{n\log^{2d+1}n}^\ell.
\end{equation}
In other words, if $\tX_n$ denotes $\tX_n:=(X_n-nv_1)/n\qq$ when the urn is strictly small and $\tX_n:=(X_n-nv_1)/\sqrt{n\log^{2d+1}n}$ when the urn is critically small, then
$\E |\tX_n|^{2\ell}=O(1)$, for every positive integer $\ell$.
Consequently, if $0\le p<2\ell$, then 
the sequence $\E |\tX_n|^{p}$ is uniformly integrable.
Since $\ell$ is arbitrary, this sequence is uniformly integrable for every
fixed $p\ge0$.
Furthermore, by \cite[Theorems 3.22 and 3.23]{SJ154}, 
$\tX_n\dto N(0;\gS)$, for some
covariance matrix $\gS$. The uniform integrability just shown implies that
any mixed moment $\E\tX_n^\ga$ converges to the corresponding moment of
$N(0,\gS)$. 
\end{proof}

\begin{proof}[Proof of \refC{C1}]
The estimates for $\E Y_n$ and $\Var Y_n$ follow directly from the results
for $\E X_n$ and $\Varx{X_n}$ in  Theorem \ref{T<} or \ref{T=}.
Furthermore, \eqref{small} yields
\begin{equation}\label{pc1}
  \frac{Y_n-n\gl_1\wprod{v_1}}{\sqrt{n\log^\nu n}}\dto N(0,\gam),  
\end{equation}
and \eqref{c1} follows when $\gam\neq0$.
Moreover, the moment convergence in \eqref{small} asserted in Theorems
\ref{T<} and \ref{T=} implies moment convergence
in \eqref{pc1}, and thus in \eqref{c1}.
\end{proof}

\newcommand\AMS{Amer. Math. Soc.}
\newcommand\Springer{Springer-Verlag}
\newcommand\Wiley{Wiley}

\newcommand\vol{\textbf}
\newcommand\jour{\emph}
\newcommand\book{\emph}
\newcommand\inbook{\emph}
\def\no#1#2,{\unskip#2, no. #1,} 
\newcommand\toappear{\unskip, to appear}

\newcommand\arxiv[1]{\texttt{arXiv:#1}}
\newcommand\arXiv{\arxiv}

\def\nobibitem#1\par{}

\end{document}